\mag=\magstep1 
\input epsf 
\documentclass[10pt,a4paper]{amsart} 
\usepackage{amssymb,latexsym} 
\usepackage[dvipdfm]{graphicx,color}

\def\Cal{\mathcal}

\def\<<{\langle } 
\def\>>{\rangle } 
\def\C{\mathcal C}

\numberwithin{equation}{section} 
\newtheorem{theorem}{Theorem}[section] 
\newtheorem{proposition}[theorem]{Proposition} 
 
\newtheorem{definition}[theorem]{Definition} 
 
\newtheorem{remark}[theorem]{Remark} 
\newtheorem{lemma}[theorem]{Lemma}

\def\<{\langle} 
\def\>{\rangle}

\begin{document} 
 
\title{
Degenerations of triple coverings and Thomae's formula 
} 
 
\author{Keiji Matsumoto and Tomohide Terasoma} 


\dedicatory{Dedicated to Professor Toshiyuki Katsura on his 60th birthday}
\maketitle 

\makeatletter 
\renewcommand{\@evenhead}{\tiny \thepage \hfill 
\hfill}

\renewcommand{\@oddhead}{\tiny \hfill 
 \hfill \thepage} 
\begin{abstract}
In this paper, we prove Thomae's formula for 
triple coverings of the complex projective line 
$\bold P^1$ and give
the absolute constant in this formula for a specific choice 
of symplectic bases.
This formula gives a relation between theta constants,
the products of 
the determinant of a period matrix and difference products
of branch points. To specify symplectic bases of them,
we use the combinatorics of binary trees on $\bold P^1$.
These symplectic bases behave so well for
degenerations that we 
reduce the formula to a special case
treated in \cite{BR}, \cite{N}.

\end{abstract}
\setcounter{footnote}{1}
\par\noindent

\section{Introduction}

Let $E:y^2=x(1-x)(1-\lambda x)$ be an elliptic curve
and $\tau$ be the normalized period matrix of $E$.
Then we have Jacobi's formula for
an elliptic curve $E$:
$$
\vartheta(\tau)[
0, 0]^2=\frac{1}{\pi}
\int_{0}^1\frac{1}{\sqrt{x(1-x)(1-\lambda x)}}dx,
$$
where $\vartheta(\tau)[\alpha, \beta]$ is defined in 
(\ref{def of theta}).

In 1870, Thomae (\cite{T}) generalized this formula to
those for hyperelliptic curves of arbitrary genus.
He showed that the squares of theta constants
at the normalized period matrix of a hyperelliptic curve
are equal to the products of the determinant of
a period matrix and certain difference products of branch points
up to an easy constant multiple.
Bershadsky-Radul and Nakayashiki \cite{BR},\cite{N} 
independently proved an analogous formula 
for cyclic coverings of the projective
line $\mathbb P^1$ with special branching indices, which
is called Thomae's formula for cyclic coverings.
They prove a power of a theta constant with charcteristic $\Lambda$ is the product of
the determinant of a period matrix, a certain
difference product and a constant $\kappa_{\Lambda}$,
and show that the constant $\kappa_{\Lambda}$ depends only on the genus,  
a choice of symplectic basis of the covering and the theta characteristic $\Lambda$.
This result is generalized to arbitrary branching indices and covering degrees by \cite{K}.

For hyperelliptic curves,
Fey computed the absolute constants $\kappa_{\Lambda}$ 
in Thomae's formula
using degeneration arguments in his book \cite{F}.
In this paper, we give a closed formula for the absolute constants $\kappa_{\Lambda}$
for triple coverings of arbitrary branching
indices.
To formulate the exact statement of Thomae's formula,
we construct symplectic bases of a family of triple coverings
using the combinatorics of binary trees on $\bold P^1$.
This family is extended to stable curves with trivial monodromy action
and each of its special fibers is the union of two triple coverings of $\bold P^1$'s.
By this degeneration, a binary tree decomposes to
two trees and according to this decomposition, the symplectic bases
are extended to the union of symplectic bases of
two irreducible components.
We use this property for the study of the absolute constant $\kappa_{\Lambda}$.

The contents of this paper are as follows.
In Section 2, we recall results of Bershadsky-Radul and Nakayashiki.
In Section 3, we define a specific choice of symplectic basis
$\{A_1, \dots, A_g, B_1,\cdots, B_g\}$ associated to
a planar binary tree. We study the combinatorial process of 
degenerations.
In Section 4, we study stable degenerations of algebraic curves
associated to the decomposition of binary trees.
In Section 5, 
we prove Thomae's formula (Theorem \ref{main theorem}) for triple coverings
of $\mathbb P^1$ and compute the absolute constants
for arbitrary branching indices using degeneration argument.
The proof is based on results of Bershadsky-Radul and
Nakayashiki and the formula 
(\ref{variant of chowla-selberg}), 
which is a variant of 
the Chowla-Selberg formula. 
Our method is different from that of \cite{K}.

{\bf Acknowledgment}
After finishing this work, we are pointed out about the references
\cite{A}, \cite{K} by Y. Kopeliovich.
The authors would like to express their thanks to him.

\section{Result of Bershadsky-Radul-Nakayashiki}
In this section, we recall results of Bershadsky-Radul and
Nakayashiki.
Let $n\geq 2$ be an integer,
$\mathbb P^1$ be the projective line with a coordinate $x$,
and $\Sigma=\{\lambda_1, \cdots, \lambda_{3n}\}$
be a set of distinct $3n$ points in $\mathbb P^1$ 
different from $x=\infty$.
The values of $x$ at $\lambda_1,\dots \lambda_{3n}$ are also denoted by the
same letter. Let $C$ be the cyclic
triple covering of $\mathbb P^1$ defined by
$$
y^3=(x-\lambda_1)\cdots (x-\lambda_{3n}).
$$
By Hurwitz's formula, the genus $g=g(C)$ of $C$ is 
equal to $3n-2$ and 
a basis of the space of holomorphic differential forms 
on $C$ is given by
$\{\omega_1,\dots, \omega_{3n-2}\}$, where
\begin{align*}
\omega_i&=\frac{x^{n-1-i}dx}{y} \qquad \text{ for }i=1, \dots, n-1, \\
\omega_{i+n-1}&=\frac{x^{2n-1-i}dx}{y^2} \qquad \text{ for }i=1, \dots, 2n-1.
\end{align*}

We fix a symplectic basis $\{A_i, B_i\}_{i=1, \dots, 3n-2}$ of 
$H_1(C,\mathbb Z)$. The period matrices $P_A$ and $P_B$
are defined as
$$
P_A=(\int_{A_i}\omega_j)_{i,j=1, \dots, 3n-2},\quad
P_B=(\int_{B_i}\omega_j)_{i,j=1, \dots, 3n-2}.
$$
The normalized period matrix $\tau$ is defined by 
$\tau=P_A{P_B}^{-1}$. 
For vectors $\alpha, \beta \in \mathbb Q^g$,
we define the theta function as follows: 
\begin{equation}
\label{def of theta}
\vartheta(\tau)
\left[\alpha , \beta\right](z)=
\sum_{m\in \mathbb Z^g}
\bold e(\frac{1}{2}(m+\alpha)\tau\ ^t(m+\alpha)
+(m+\alpha)\ ^t(z+\beta)).
\end{equation}
Here, we use the notation $\bold e(x)=\exp(2\pi \sqrt{-1} x)$.
The value of the theta function at $z=0$ is called the theta constant 
and denoted by 
$\displaystyle
\vartheta(\tau)
\left[ \alpha+a, \beta+b\right]$.
Note that if $\alpha,\beta\in\frac{1}{6}\mathbb Z^g$, then
the sixth power of the theta constant is 
periodic with respect to $\alpha$ and $\beta$,
that is
$$
\vartheta(\tau)
\left[ \alpha , \beta\right]^6=
\vartheta(\tau)
\left[ \alpha+a , \beta+b\right]^6
$$
for $a,b\in \mathbb Z^g$.
By the isomorphism 
\begin{equation}
\label{identification of homology}
\mathbb Q^g\oplus\mathbb Q^g\to H_1(C,\mathbb Q):
(a_1,\cdots, a_g,b_1, \dots, b_g)\mapsto 
\sum_ia_iA_i+\sum_ib_iB_i,
\end{equation}
the theta constant $\displaystyle\vartheta (\tau)
\left[ \alpha, \beta\right]^6$
is a function on the set of $6$-torsion points 
$H_1(C,\mathbb Q/\mathbb Z)_6$ of $H_1(C,\mathbb Q/\mathbb Z)$.
The value at $\lambda\in H_1(C,\mathbb Q/\mathbb Z)_6$
is denoted by $\vartheta(\tau)[\lambda]$.

Let $\rho:C\to C$ be the automorphism of the curve $C$
defined by $(x,y)\mapsto(x, \omega y)$,
where $\omega=\displaystyle \frac{-1+\sqrt{-3}}{2}$.
The induced homomorphism on $H_1(C,\mathbb Z)$ is also denoted by
$\rho$.
We give a description of the group of $(1-\rho)$
torsion part $H_1(C, \mathbb Q/\mathbb Z)_{(1-\rho)}$
of $H_1(C, \mathbb Q/\mathbb Z)$ as a $\mathbb F_3$-vector
space.

Let $b$ be a base point in $\mathbb P^1-\Sigma$. A
path connecting $b$ and small anti-clockwise circle 
around $\lambda_i$
defines a path $\overline{\gamma_i}$ in $\mathbb P^1-\Sigma$. We choose a
base point $\tilde b$ in $C$ over the base point $b$. Then the path
$\overline{\gamma_i}$ is lifted uniquely beginning from $\tilde b$,
and its end point is $\rho(\tilde b)$.
If any paths $\gamma_i-\{\tilde b,\rho(\tilde b)\}$ are disjoint 
to one other, then we may assume that
paths $\gamma_1, \dots, \gamma_{3n}$ 
are arranged in the anti-clockwise order
by renumbering the paths 
$\gamma_1, \dots, \gamma_{3n}$.
Then the cycle
$$
\gamma_1+\rho\gamma_2+\rho^2\gamma_3+\cdots +\rho^{3n-1}\gamma_{3n}
$$
is homologous to zero. 
For $i=1,\dots, 3n-1$, the path $\gamma_i-\gamma_{3n}$ defined an
element of $H_1(C,\mathbb Z)$. Therefore 
$\delta_i=\frac{1}{3}(1-\rho)(\gamma_i-\gamma_{3n})$ defines an element of
$H_1(C, \mathbb Q/\mathbb Z)_{(1-\rho)}$.
By the above relation, the cycle
$\delta_1+\rho\delta_2+\cdots \rho^{3n-2}\delta_{3n-1}$ is homologous to zero.
Let $\bold V$ be the $\mathbb F_3$-vector space $\oplus_{i=1}^{3n}\mathbb F_3e_i $ 
generated by $e_1,\dots, e_{3n}$. We define the map 
$\Pi$ and $\Delta$ by
\begin{align*}
&\Pi:\bold V  \ni \sum_i a_ie_i \mapsto \sum_i a_i \in\mathbb F_3, \\
&\operatorname{Diag}:
\mathbb F_3 \ni a \mapsto a\sum_{i=1}^{3n} e_i \in \bold V.
\end{align*}
Then we have an isomorphism
$$
H_1(C, \mathbb Q/\mathbb Z)_{(1-\rho)}\simeq Ker(\Pi)/Im(
\operatorname{Diag})
$$
by assigning $\delta_i$ to $e_i-e_{3n}$.

We are ready to state Thomae's formula for the curve $C$ 
(\cite{N},\cite{BR}). Let $\widetilde\Lambda=\sum_{i}a_ie_i$ be a representative
in $\bold V$ of an element $\Lambda$ of
$H_1(C,\mathbb Q/\mathbb Z)_{(1-\rho)}$. We define a subset
$\widetilde\Lambda_i$ of $\{1, \dots, 3n\}$ for $i=0,1,2$ by
$$
\widetilde\Lambda_i=\{p\in \{1, \dots, 3n\}
\mid a_p\equiv i (\text{ mod }3)\}.
$$
The difference product $(\widetilde\Lambda_i\widetilde\Lambda_j)$ 
is defined by
$$
(\widetilde\Lambda_i\widetilde\Lambda_j)=\begin{cases}
\prod_{a\in \widetilde\Lambda_i,b\in \widetilde\Lambda_j}
(\lambda_a-\lambda_b)\quad
& \text{ if }i\neq j, \\
\prod_{a,b\in \widetilde\Lambda_i,a<b}(\lambda_a-\lambda_b)\quad
& \text{ if }i=j.
 \end{cases}
$$
Then it is defined up to sign.
We define the difference product $\Delta(\Lambda)$ attached to
$\Lambda\in H_1(C,\mathbb Q/\mathbb Z)_{(1-\rho)}$ by
$$
\Delta(\Lambda)=\prod_{i=0}^2(\widetilde\Lambda_i\widetilde\Lambda_i)^3
\prod_{0\leq i<j \leq 2}(\widetilde\Lambda_i\widetilde\Lambda_j)
$$
for a representative $\widetilde\Lambda$ of $\Lambda$. 
It does not depend on 
the choice of the representative up to sign.
\begin{theorem}
Let $\Lambda$ be an element $H_1(C,\mathbb Q/\mathbb Z)_{(1-\rho)}$
and $\widetilde\Lambda$ be a representative of $\Lambda$ in $\bold V$.
Suppose that 
\begin{equation}
\label{equi-dist all white}
\# \widetilde\Lambda_0=\#\widetilde\Lambda_1
=\#\widetilde\Lambda_2=n.
\end{equation}
Then we have
$$
\vartheta(\tau)[\Lambda+\varrho]^6=\kappa_{\Lambda}\det(P_B)^3\cdot \Delta(\Lambda). 
$$
Here $\varrho$ is the Riemann constant and $\kappa_{\Lambda}$ is a constant 
independent of 
$\lambda_1,\dots, \lambda_{3n}$. Moreover $\kappa_{\Lambda}^6$ is independent
of $\Lambda$.
\end{theorem}
\begin{definition}
The absolute constant $\kappa_{\Lambda}^6$ is called Thomae's
constant for the symplectic basis $\{A_i,B_i\}_i$.
\end{definition}
\section{Binary trees and symplectic bases}
\subsection{Symplectic basis associated to a marked binary tree}
In this section, we define a symplectic basis associated 
to a binary tree and study its properties.
For general terminology on trees, refer to \cite{S} and \cite{D}.
\begin{definition}
\begin{enumerate}
\item
A vertex of a tree $\Gamma$ is called an inner vertex if
it is adjacent to more than one edges.
The set of inner vertices of $\Gamma$ is denoted by $V(\Gamma)$.
A vertex of $\Gamma$ is a terminal if it is not a inner vertex.
\item
A tree is called a trivalent tree if the valency of any inner vertex
is three.
\item
A planar trivalent tree
$\Gamma$ 
with two colors, white and black, on its vertices
is called a binary tree if two ends of
any edge are differently colored. 
\item
Let $(\Gamma,\C)$ be a binary tree.
A marking of $(\Gamma,\C)$ is to specify
choices of one edge adjacent to each 
inner vertex.
A binary tree $(\Gamma, \C)$ with a marking $M$ is 
called a marked binary tree and denoted by $(\Gamma,\C, M)$.
We connect the rest of edges by a small arc to illustrate
the marking. (See Figure \ref{pict marked binary tree} 
for an example of marked
binary tree.)
\end{enumerate}
\end{definition}

\begin{figure}[hbt]
\hskip 0.0in\includegraphics[scale=0.6]{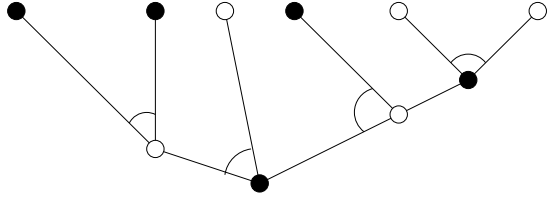}
\caption{Marked binary tree}
\label{pict marked binary tree}
\end{figure}

Let $\lambda_1,\dots, \lambda_m$ be points in $\mathbb P^1$
and set $\Sigma=\{\lambda_1, \dots, \lambda_m\}$.
We choose a base point $b$ in $\mathbb P^1-\Sigma$ and
a coordinate $t$ so that the base point $b$ corresponds to
the infinity. Let $(\Gamma,\C)$ be a binary tree
on $\mathbb P^1$, whose set of terminals is equal
to $\{\lambda_1, \dots, \lambda_m\}$. 
We construct a cyclic triple covering $C^0$ of $\mathbb P^1-\Sigma$
associated to the binary tree $(\Gamma, \C)$.
We prepare three copies $C^{(0)},C^{(1)},C^{(2)}$ of 
$\mathbb P^1-\Gamma$. We patch them according to the following
rule. (See Figure \ref{patching of sheets}.)
\begin{enumerate}
\item
If a point moves across an edge of $\Gamma$ looking white
vertex left, we change the sheet as
$C^{(0)}\to C^{(1)}\to C^{(2)}\to C^{(0)}$.  
\item
If a point moves across an edge of $\Gamma$ looking white
vertex right, we change the sheet as
$C^{(0)}\to C^{(2)}\to C^{(1)}\to C^{(0)}$.  
\end{enumerate}

\begin{figure}[hbt]
\hskip 0.0in\includegraphics[scale=0.4]{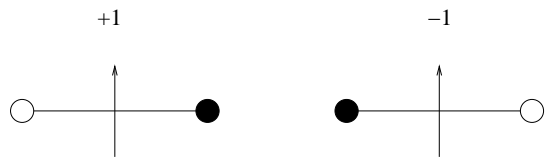}
\caption{Patching of sheets}
\label{patching of sheets}
\end{figure}

In this way, we get a compact Riemann surface
$C$.
Since the tree $\Gamma$ is trivalent, the covering 
$C \to \mathbb P^1$ is not branching at any point $v$ in
$V(\Gamma)$, and
the branching index $a_i/3$ of $\lambda_i$ in $\Sigma$ is equal to
$1/3$ or $2/3$ if the color of the vertex $\lambda_i$ is
white or black, respectively.
Then the Riemann surface $C$ is isomorphic to the compactification of the 
Kummer covering of $\mathbb P^1$:
$$
y^3=\prod_{i=1}^m(x-\lambda_i)^{a_i}.
$$
The automorphism $\rho$ defined by 
$(x,y)\mapsto (x,\omega y)$
corresponds to the change of sheets
$C^{(0)}\to C^{(1)}\to C^{(2)}\to C^{(0)}$.  

\begin{lemma}
The genus $g=g(C)$ is equal to the number $\# V(\Gamma)$
of $V(\Gamma)$.
\end{lemma}
\begin{proof}
Since the tree is trivalent, the number of edges is
equal to 
$(3\# V(\Gamma)+\#\Sigma)/2$.
Therefore the number of the vertices is
$$
(3\# V(\Gamma)+\#\Sigma)/2+1=\# V +\#\Sigma.
$$
Therefore we have
$\#V(\Gamma)+2=\#\Sigma$
and 
$
g=\#\Sigma -2=\#V(\Gamma).
$
\end{proof}
For a marking $M$ of the binary tree $(\Gamma, \C)$,
we define a symplectic basis $\{A_v,B_v\}_{v\in V(\Gamma)}$ 
as follows.
We define cycles $A_v$ and $B_v$ as follows.
By deleting a small neighborhood of $v$ from the tree
$\Gamma$, we decompose it into three blocks 
$\bold B1,\bold B2$ and $\bold B3$
as in Figure \ref{picture of cycles}.
We define a cycle $A_v$ by the path in Figure 
\ref{picture of cycles}.

\vskip 0.1in
\begin{figure}[htb]
\includegraphics[scale=0.3]{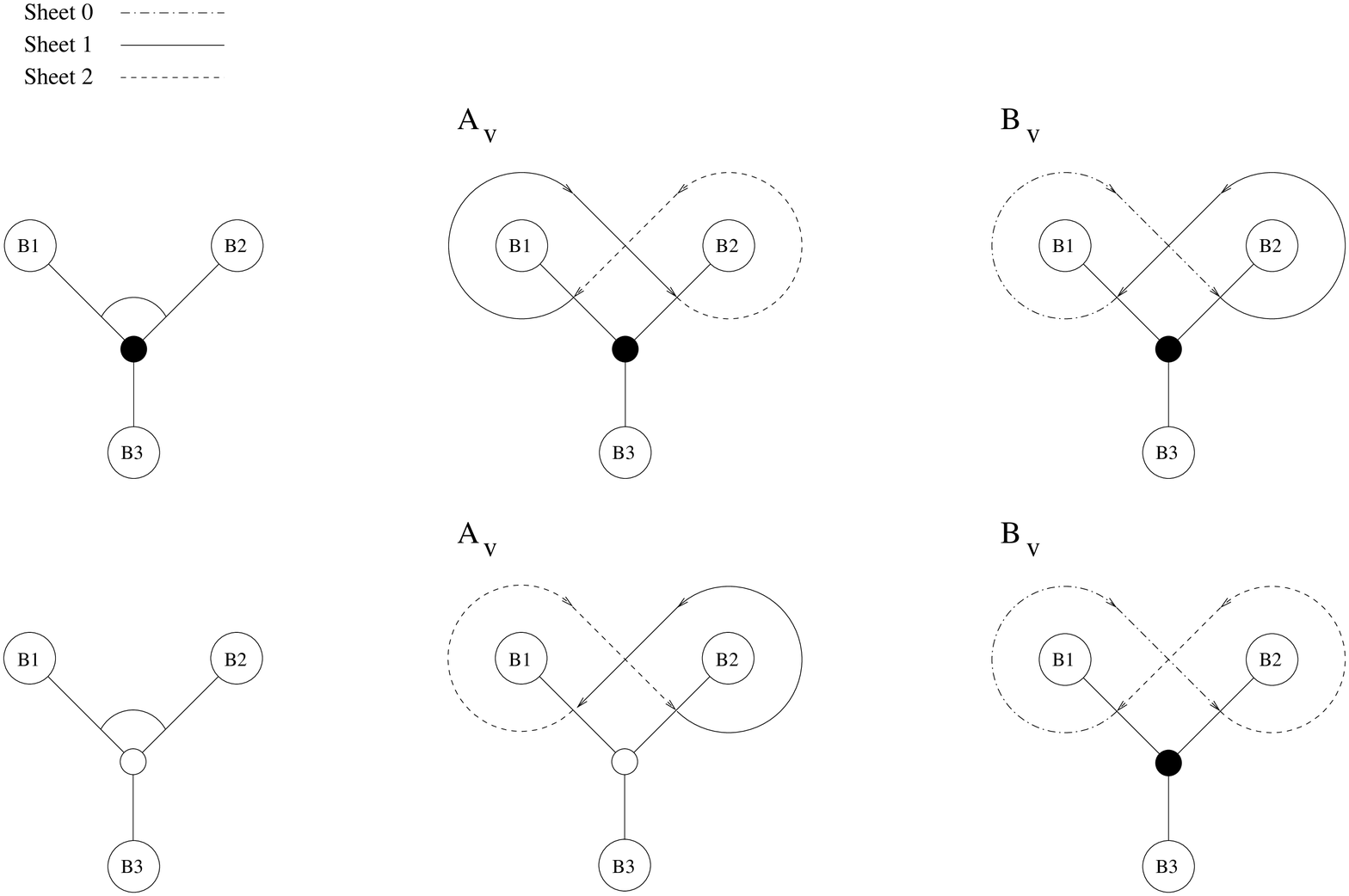}
\caption{Cycles}
\label{picture of cycles}
\end{figure}

If the vertex $v$ is white (resp. black), 
we define $B_v$ by $B_v=\rho^2(A_v)$ (resp. $B_v=\rho(A_v)$).
It is easy to see the following proposition.
\begin{proposition}
The set of cycles $\{A_v,B_v\}_{v\in V(\Gamma)}$
forms a symplectic basis. That is
$$
(A_v,A_w)=(B_v,B_w)=0, \qquad (A_v,B_w)=-\delta(v,w)
$$
for $v,w \in V(\Gamma)$.
\end{proposition}

\subsection{Riemann constants}
\label{subsect:Riemann constants}
In this subsection, we describe the Riemann constant
of the curve $C$ with the symplectic basis
$\{A_v, B_v\}_{v\in V(\Gamma)}$ constructed in the last subsection.
Let $\tau$ be the normalized period matrix with respect to
$\{A_v,B_v\}_{v \in V(\Gamma)}$ and $p_0$ be one of 
the branched points of the covering $C\to \bold P^1$.
Let $S^{g-1}(C)$ be the $(g-1)$-th symmetric product of $C$,
$Q=\sum_{i=1}^{g-1}q_i$ be an element in $S^{g-1}$ and
$sym^{g-1}_{Q}:S^{g-1}(C)\to Jac(X)$ be the map defined by
$$
S^{g-1}(C) \ni \sum_{i=1}^{g-1}p_i \mapsto 
\sum_{i=1}^{g-1}(p_i)-(q_i) \in Pic(C).
$$
We assume that $q_i$ are branch points and $2Q=K_C$.
We define the theta zero divisor by
$$
(\vartheta(\tau)[0,0](\tau,z))_0=
\{[z]\in Jac(X)\mid \vartheta(\tau)[0,0](\tau,z)=0\}.
$$
Then by Riemann's theorem, there is a point 
$\varrho(p_0,\{A_i,B_i\})$ in $Jac(X)$
such that
\begin{align*}
(\vartheta(\tau)[0,0](\tau,z))_0=
\varrho(Q,\{A_i,B_i\}) -sym_{Q}^{g-1}(S^{g-1}(C)).
\end{align*}
The point $\varrho$ is called the Riemann constant.
By \cite{Mu},I,Cor. 3.1.1,p.166, we have 
$2\varrho(Q,\{A_i,B_i\})=0$ in $Jac(C).$
Let $\tau'$ be the normalized period matrix with respect
to the symplectic basis $\{\rho(A_i),\rho(B_i)\}$.
By the theta transformation formula, we have
$$
(\vartheta(\tau)[0,0](\tau,z))_0-
(\vartheta(\tau)[0,0](\tau',z))_0=m
$$
where
$$
m=\sum_{v\in \text{ black vertex} }\frac{1}{2}B_i+
\sum_{v\in \text{ white vertex} }\frac{1}{2}A_i.
$$
Thus, we have
\begin{align*}
m&=\varrho(Q,\{A_i,B_i\})-
\varrho(Q,\{\rho(A_i),\rho(B_i)\}) 
=(1-\rho)\varrho(Q,\{A_i,B_i\}).
\end{align*}
Therefore in the group of two torsion points in
$Jac(X)$, we have
\begin{align*}
\varrho(Q,\{A_i,B_i\})
=(1-\rho^2)m 
=\sum_{v}\frac{1}{2}(A_i+B_i). 
\end{align*}

\subsection{Description of $(1-\rho)$ torsion cycles}

We use the same notation as in the last subsection.
We fix a numbering $\lambda_1, \cdots, \lambda_m$ compatible
with the cyclic ordering
on the set of the terminals of the binary tree $\Gamma$.
Let $\overline{\gamma_i}$ be a loop in $\mathbb P^1-\Sigma$ 
starting 
the base point $b\in \mathbb P^1$ which turns around
$\lambda_i$ once anti-clockwise.
The point in $C^{(i)}$ over the point $b$
is denoted by $b(i)$ for $i=0,1,2$.
The lift of $\overline{\gamma_i}$
starting from $b(1)$ (resp. $b(2)$) is denoted by $\gamma_i$ 
if the color of the vertex $\lambda_i$ is white (resp. black).
Then the path $\gamma_i$ end at the point $b(2)$ (resp. $b(1)$)
if the color of the vertex $\lambda_i$ is white (resp. black).

\begin{figure}[htb]
\hskip 0.0in\includegraphics[scale=0.4]{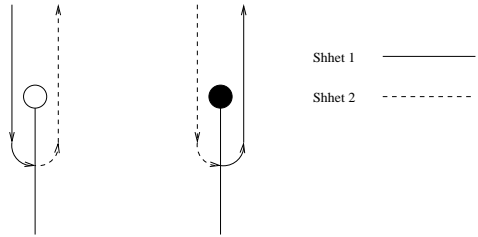}
\caption{Path $\gamma_i$}
\end{figure}
Then by choosing suitable integers $k_1, \dots, k_m$,
we have a relation 
\begin{equation}
\label{boundary relation}
\rho^{k_1}\gamma_1+\rho^{k_2}\gamma_2+ \cdots + \rho^{k_m}\gamma_m=0
\end{equation}
in the group $H_1(C, \mathbb Z)$.

The sign $\epsilon_i$ of a terminal $\lambda_i$
is assigned as $1$ or $-1$ 
if its color is white or black, respectively. 
Then the boundary of the chain $\epsilon_i\gamma_i$ is equal to 
$b(2)-b(1)$.
We define cycles $\delta_1, \dots, \delta_{m-1}$ by
$$
\delta_i=\epsilon_i\gamma_i - \epsilon_m\gamma_m.
$$
By the equation (\ref{boundary relation}), we have
\begin{equation}
\label{cycle relation}
\rho^{k_1}\epsilon_1\delta_1+
\rho^{k_2}\epsilon_2\delta_2+ \cdots + 
\rho^{k_{m-1}}\epsilon_{m-1}\delta_{m-1} =0.
\end{equation}
The element in $H_1(C,\mathbb Q/\mathbb Z)_{(1-\rho)}$
represented by
$\frac{1}{3}(1-\rho)\delta_i$ is denoted by
$\overline{\delta_i}$.

Here we give a combinatorial description of the group 
$H_1(C,\mathbb Q/\mathbb Z)_{(1-\rho)}$. Let $\bold V$ be the 
$\mathbb F_3$-vector space
$\oplus_{i=1}^m e_i\mathbb F_3$ generated by $e_1, \dots, e_m$.
We define homomorphisms $\Pi$ and $\operatorname{Diag}$ by
\begin{align}
\label{def of pi}
&\Pi:\bold V \ni \sum_ik_ie_i\mapsto \sum_i \epsilon_i k_i
\in \mathbb F_3,  \\
\nonumber
&\operatorname{Diag}:\mathbb F_3
\ni a \mapsto a\sum_i e_i \in \bold V.
\end{align}
Then the set $\{\epsilon_ie_i-\epsilon_me_m\}_{i=1, \dots, m-1}$ 
is a basis of $Ker(\Pi)$. 
By assigning the class of $\overline{\delta_i}$
in $H^1(C, \mathbb Q/\mathbb Z)_{(1-\rho)}$
to $\epsilon_ie_i-\epsilon_me_m$, we have a map
$Ker(\Pi) \to H^1(C, \mathbb Q/\mathbb Z)_{(1-\rho)}$ and as a consequence,
we have an isomorphism
\begin{equation}
\label{combinatorial isom of h1}
Ker(\Pi)/Im(\operatorname{Diag}) \overset{\simeq}\to
H^1(C, \mathbb Q/\mathbb Z)_{(1-\rho)}.
\end{equation}
They are $(m-2)$-dimensional vector spaces.
\begin{definition}
The quotient space $Ker(\Pi)/Im(\operatorname{Diag})$
is obtained from the information of the binary tree $(\Gamma,\C)$.
It is denoted by
$H(\Gamma,\C)$.

\end{definition}

\subsection{Combinatorial computation of $\displaystyle\frac{1-\rho}{3}A_v$}
In this subsection, we give a combinatorial rule to compute
the class $\overline{A_v}$ of 
$\displaystyle\frac{1-\rho}{3}A_v$ in $H_1(C,\mathbb Q/\mathbb Z)_{(1-\rho)}$.
We illustrate an element $\Lambda=\sum_i k_ie_i$
by writing $k_i$ to each
terminal $\lambda_i$ in the tree $\Gamma$.
(For example, see Figure \ref{pict the simplest case}.)

First, we consider the simplest case: $m=3$ and the color 
of $\lambda_i$ is white for $i=1,2,3$. There is only one inner black
vertex. We assume that $\lambda_3$ is marked. 
\begin{figure}[hbt]
\includegraphics[scale=0.3]{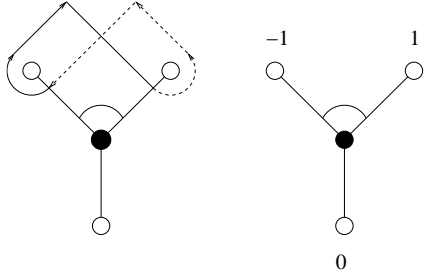}
\caption{Simplest case}
\label{pict the simplest case}
\end{figure}
Then we have $A_v=\gamma_2-\gamma_1$ as in 
Figure \ref{pict the simplest case}.
Via the isomorphism (\ref{combinatorial isom of h1}),
we have $\overline{A_v}=e_2-e_1$. 
%
Similarly we can compute the image of $\overline{A_v}$
in the case: $m=3$ and the color of $\lambda_i$ is black.
In this case, we also have $\overline{A_v}=e_2-e_1$.
Note that the class of $\rho(\displaystyle\frac{1-\rho}{3}A_v)$
is equal to $\overline{A_v}$ in each case.

Next, we consider a general situation. 
Let $v$ be a black vertex connected to three blocks
$\bold B_1,\bold B_2,\bold B_3$. 
(A block may be one point.)
Recall that the cycle $A_v$ is given in
Figure \ref{General case}. 
\begin{figure}[hbt]
\includegraphics[scale=0.3]{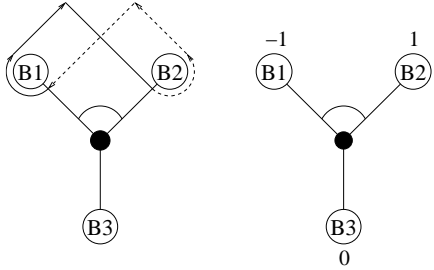}
\caption{General case}
\label{General case}
\end{figure}
We reduce the
computation $\overline{A_v}$ to the following ``local deformation rule''.
Figure \ref{block pict} is a local picture
contained in $\bold B_2$.
The blocks connected to an inner vertex $p$ are
denoted by $\bold C_1,\bold C_2$ and $\bold C_3$.
Let $\gamma$ be a path starting from
$b(2)$ ending at $b(1)$ turning around $\bold C_1$ and $\bold C_2$
as in Figure \ref{block pict}. Note that
the path
$\gamma$ is homotopic to $\rho(\gamma_{\bold C_1})
+\rho^2(\gamma_{\bold C_2})$, 
where $\gamma_{\bold C_i}$ is a path turning around the block 
$\bold C_i$ anti-clockwise.
Similarly, if $p$ is a white vertex, $\gamma$ is homotopic
to $\rho^2(\gamma_{\bold C_1})+\rho(\gamma_{\bold C_2})$. 
\begin{figure}[htb]
\hskip 0.0in\includegraphics[scale=0.4]{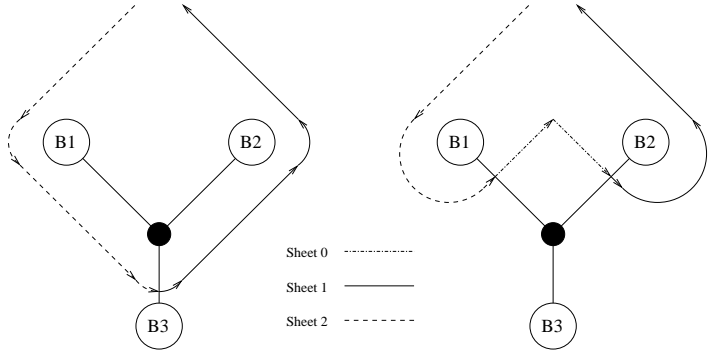}
\caption{Local deformation}
\label{block pict}
\end{figure}
By this local deformation rule, 
the coefficient of $e_i$ in $\overline{A_v}$
is equal to $-1,1$ and
$0$ if the terminal $\lambda_i$ belongs to $\bold B_1, 
\bold B_2$ and $\bold B_3$, respectively.

Finally, we have the following proposition.
\begin{proposition}
\begin{enumerate}
\item
The class $\overline{A_v} \in H_1(C,\mathbb Q/\mathbb Z)_{(1-\rho)}$
of $\frac{1}{3}(1-\rho)A_v$ is equal to the element
$$
-\sum_{\lambda_i\in \bold B_1\cap \Sigma}e_i
+\sum_{\lambda_i\in \bold B_2\cap \Sigma}e_i
=-\sum_{\lambda_i\in \bold B_2\cap \Sigma}e_i
+\sum_{\lambda_i\in \bold B_3\cap \Sigma}e_i
=-\sum_{\lambda_i\in \bold B_3\cap \Sigma}e_i
+\sum_{\lambda_i\in \bold B_1\cap \Sigma}e_i
$$
via the isomorphism (\ref{combinatorial isom of h1}).
\item
If $v$ is a black vertex (resp. a white vertex), then 
$\overline{A_v}\equiv\frac{1}{3}(-A_v+B_v)$ mod $H_1(C,\mathbb Z)$
(resp.
$\overline{A_v}\equiv\frac{1}{3}(A_v-B_v)$ mod $H_1(C,\mathbb Z)$
).
\item
The set $\{\overline{A_v}\}_{v\in V(\Gamma)}$ forms a basis of 
$H_1(C,\mathbb Q/\mathbb Z)_{(1-\rho)}$.
\end{enumerate}
\end{proposition}

\subsection{A decomposition of a marked binary tree}
In this subsection, we consider a decomposition of a marked tree.
Let $(\Gamma,\C,M)$ be a marked binary tree and $E$ be its edge.
Let $p$ and $q$ be white and black vertices adjacent to $E$,
respectively.
Let $\bold C_p$ and $\bold C_q$ be two blocks containing $p$ and $q$,
respectively. We define subspaces $H_p$ and $H_q$ of
$H(\Gamma,\C)$ by
\begin{align*}
&H_p=\{\sum_{v\in \Sigma} a_v e_v
\in H(\Gamma,\C)
\mid
a_v=a_{v'} \text{ for }v, v' \in \bold C_q\cap \Sigma \}, \\
&H_q=\{\sum_{v\in \Sigma} a_v e_v
\in H(\Gamma,\C)
\mid
a_v=a_{v'} \text{ for }v, v' \in \bold C_p\cap \Sigma \}. 
\end{align*}
We remark that the condition in the definition of 
$H_p$ and $H_q$
does not depend on the representative modulo 
$Im(\operatorname{Diag})$.
By cutting the edge $E$ of tree and inserting black and white vertices
to $\bold C_p$ and $\bold C_q$, we get marked binary trees 
$(\Gamma_p,\C_p,M_p)$ and
$(\Gamma_q,\C_q,M_q)$ as in the following picture. 

\vskip 0.1in
\begin{figure}[htb]
\hskip 0.0in\includegraphics[scale=0.4]{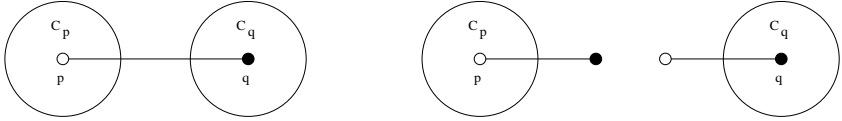}
\caption{Decomposition of a marked binary tree}
\end{figure}
It is easy to see the following lemma.
\begin{lemma}
\begin{enumerate}
\item
The element $\overline{A_v}$ belongs to the space
$H_p$ and $H_q$ if $v$ belongs to $\bold C_p$ and $\bold C_q$, 
respectively.
\item
The space $H_1(C,\mathbb Q/\mathbb Z)_{(1-\rho)}$
is the direct sum of $H_p$ and $H_q$. The sets
$\{\overline{A_v}\}_{v\in V(\Gamma_p)}$ and 
$\{\overline{A_v}\}_{v\in V(\Gamma_q)}$ 
form bases of $H_p$ and $H_q$, respectively.
\item
The vector spaces $H_p$ and $H_q$ are naturally isomorphic
to $H(\Gamma_p,\C_p)$ and $H(\Gamma_q,\C_q)$, respectively.
\end{enumerate}
\end{lemma}

\section{Stable degeneration of triple coverings of $\mathbb P^1$}
In this section, we compare a decomposition of a marked binary tree
and a stable degeneration of triple coverings of $\mathbb P^1$.

Let $\lambda_1,\dots, \lambda_m$ be points in $\mathbb C-\{0\}$.
Let $(\Gamma, \C, M)$ be a marked binary tree whose set of terminals
is $\{\lambda_1,\cdots ,\lambda_m\}$. 
The branching indices of $\lambda_1, \dots, \lambda_m$ are written
as $a_1/3, \cdots, a_m/3$.
We choose an edge $E$ and obtain 
two blocks $\bold C_p,\bold C_q$ by cutting $E$ as in 
the last section. By renumbering the set of terminals,
we can assume that 
$\bold C_p\cap \Sigma =\{\lambda_1, \dots,\lambda_l\}$ and 
$\bold C_q\cap \Sigma =\{\lambda_{l+1},\cdots, \lambda_m\}$
and that the cyclic order
of the terminals is equal to  $\lambda_1, \dots, \lambda_m$.
We consider a deformation of the set of terminals parametrized by $t$
as
$$
\lambda_i(t)=t\lambda_i \quad (i=1, \dots, l),\qquad
\lambda_i(t)=\lambda_i \quad (i=l+1, \dots, m).
$$
Suppose that $\lambda_1, \dots, \lambda_l$ are sufficiently close to $0$.
We consider a family of triple coverings $C(t)$ of $\mathbb P^1$
parametrized by $t\in \Delta^*(t)=\{t\in \mathbb C\mid 0<\mid t\mid<1\}$ as
\begin{equation}
\label{deformation of equation}
y^3=\prod_{i=1}^m(x-\lambda_i(t))^{a_i}.
\end{equation}
This family extends to a stable model by the base change
$\varepsilon^3=t$. We write down an affine birational model as follows. 
Let $\Cal X$ be a submanifold of $\mathbb C^1\times \mathbb C^1\times 
\Delta(\varepsilon)$
defined by
$$
\Cal X=\{(x, \xi, \varepsilon)\in 
\mathbb P^1\times \mathbb P^1\times \Delta\
\mid x\xi=\varepsilon^3\}.
$$
Let $a, b$ be elements in 
$\{1,2\}$ satisfying $\sum_{i=1}^la_i\equiv a$ (mod 3)
and $\sum_{i=l+1}^ma_i\equiv b$ (mod 3).
Let $\Cal Y$ be the triple covering of $\Cal X$ defined by
new coordinates $y_1, y_2$ satisfying
\begin{align}
\label{first component}
& y_1^3=x^a\prod_{i=1}^l(1-\lambda_i\xi)^{a_i}
\prod_{i=l+1}^m(x-\lambda_i)^{a_i}, \\
\label{second component}
& y_2^3=\xi^b\prod_{i=1}^l(1-\lambda_i\xi)^{a_i}
\prod_{i=l+1}^m(x-\lambda_i)^{a_i}, 
\end{align}
with $\xi y_1=\varepsilon^ay_2,xy_2=\varepsilon^by_1$.
Then $\varphi:\Cal Y\to \Delta(\epsilon)$ is a family of affine 
curves over $\Delta(\epsilon)$.
The fiber of $\Cal Y$ at the origin $\tau=0$
is the union of two triple coverings of
$\mathbb P^1$ connected at one point. One component is the restriction of
(\ref{first component})
to $\xi=0$ and the other is that of (\ref{second component})
to $x=0$.

We define a family of marked binary trees 
$(\Gamma_p,\C_p,M_p)(t)\coprod (\Gamma_q,\C_q,M_q)(t)$
with two components parametrized by $t$.
We set $(\Gamma_q,\C_q,M_q)(t)$ as a constant family,
since $\bold C_q\cap\Sigma=\{\lambda_{l+1},\dots, \lambda_m\}$.
We set the tree $\Gamma_p(t)$ as the multiplication of the original
tree by $t$ for $t\in \Delta^*$ since 
$C_p\cap \Sigma=\{\lambda_1,\dots, \lambda_l\}$,
Its coloring and marking are naturally defined by the multiplication
by $t$. By using the coordinate $\xi$, we see that 
this tree is continued to $\varepsilon=0$.
\begin{proposition}
The monodromy around $\varepsilon=0$ acts trivially on
the family of the relative homology $R_1\varphi_*\mathbb Z$.
\end{proposition}
\begin{proof}
For a point $t=1$,
the symplectic basis $\{A_v,B_v\}_{v\in V(\Gamma)}$
associated to $(\Gamma,\C,M)$ is equal to the union of 
those associated to $(\Gamma_p,\C_p,M_p)$ and $(\Gamma_q,\C_q,M_q)$.
They are extended continuously to $\varepsilon=0$.
\end{proof}
\begin{remark}
Repeating this degeneration process, any triple covering
of $\mathbb P^1$ degenerates to the union of elliptic curves
$E_v$ ($v\in V(\Gamma)$) with complex multiplication with 
$\mathbb Z[\omega]$, which is called a totally degenerated curve.
The binary tree is equal to the dual graph of the totally degenerated
curve. The color of an inner vertex $v$ corresponds to the eigen value
of the action of $\rho$ on $H^0(E_v,\Omega^1)$.
\end{remark}

\section{Degenerations of period integrals and their determinants}
\label{section degeneration}
\subsection{Deformation of period integrals}
\label{small deformation}
In this section, we study degenerations of period integrals
for the simplest case. We use the same notations 
$\lambda_1, \dots,\lambda_m$ for branch points,
$a_1/3, \dots, a_m/3$ for branching indices and a marked binary tree 
$(\Gamma, \C, M)$. 
Therefore the color of a vertex $\lambda_i$ is white if
and only if $a_i=1$.
In this section, we assume that
$\lambda_{i}$ and $\lambda_{i+1}$ is adjacent to 
a common inner vertex
$p$. Thus the colors of $\lambda_{i}$ and $\lambda_{i+1}$
are same. The third vertex adjacent to $p$ is denoted by $q$.
Moreover we assume that the edge to $q$ is marked for the vertex $p$.
If the vertices $\lambda_{i}$ and $\lambda_{i+1}$ are white,
the situation is illustrated in Figure \ref{simple degeneration}.

\begin{figure}[hbt]
\hskip 0.4in\includegraphics[scale=0.4]{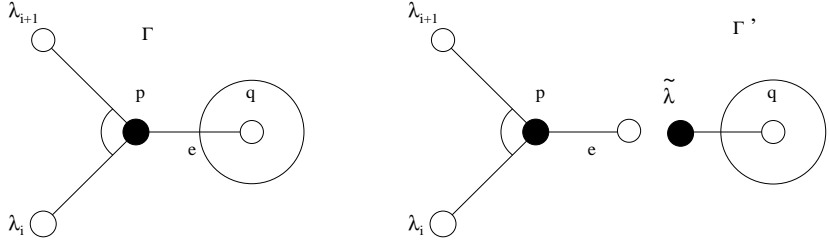}
\caption{Simple degeneration}
\label{simple degeneration}
\end{figure}

Let $\{A_v,B_v\}$ be the symplectic basis of $H_1(C,\mathbb Z)$.
Let $\tilde \lambda$ be a point different from
$\lambda_j$ for $j\neq i,i+1$ and consider the family of points
$\lambda_j(t)$ defined by $\lambda_j(t)=\lambda_j$
for $j\neq i,i+1$ and $\lambda_k(t)=\tilde\lambda+t(\lambda_k-\tilde\lambda)$
for $k=i,i+1$. We consider the family of curves
(\ref{deformation of equation})
branched at 
$\Sigma(t)=\{\lambda_1(t),\dots, \lambda_m(t)\}$.
By changing parameter $t$ to $\varepsilon$ defined as $t=\varepsilon^3$,
we get a stable degeneration associated to the above
decomposition of the binary tree as in the last section. 
The tree containing
the vertex $q$ obtained by this decomposition is
denoted by $(\Gamma', \C',M')$.

Let $y_1$ and $y_2$ be rational functions on $C$ defined by
\begin{align*}
& y_1^3=\prod_{j=1}^m(x-\lambda_j(t))^{a_j}, \\
& y_2^3=\prod_{j=1}^m(x-\lambda_j(t))^{b_j},
\end{align*}
where $b_j=3-a_j$. 
Let $d_1=\sum_j a_j-1$ and $d_2=\sum_j b_j-1$.
We define a family of differential forms
$\eta_j(t)$ for $j=1, \dots, d_1+d_2$ by
\begin{align*}
\eta_j&=\frac{x^{j-1}dx}{y_1} \quad\text{ for }j=1,\dots, d_1, \\
\eta_{j+d_1}&=
\frac{x^{j-1}dx}{y_2} \quad\text{ for }j=1,\dots, d_2.
\end{align*}
Then $\{\eta_1,\dots, \eta_{d_1+d_2}\}$
is a basis of $H^0(X,\Omega^1)$.
\begin{proposition}
Assume that the colors of $\lambda_i$
and $\lambda_{i+1}$ are white.
We set
$$
\widetilde{\eta_{j+d_1}}=
\frac{(x-\tilde\lambda)^{j-1}dx}{y_1} 
\quad\text{ for }j=1,\dots, d_2.
$$
\begin{enumerate}
\item
We have
\begin{align*}
&\lim_{t\to 0}\int_{B_p}\eta_j(t)=0 \quad 
\text{ for } j=1, \dots, d_1, \\
&\lim_{t\to 0}\int_{B_p}\widetilde{\eta_j(t)}=0 \quad 
\text{ for } j=d_1+2, \dots, d_1+d_2, 
\end{align*}
and
$$
\lim_{t\to 0}
t\cdot (\lambda_i-\lambda_{i+1})
\big(
\int_{B_p}
\widetilde{\eta_{d_1+1}(t)}\big)^3=
B^*(\frac{1}{3},\frac{1}{3})^3
\prod_{k\neq i,i+1}(\tilde\lambda-\lambda_k)^{-b_k}.
$$
Here we set $B^*(\frac{1}{3},\frac{1}{3})=(\omega-1)B(\frac{1}{3},\frac{1}{3})$.
\item
The limit
$$
\lim_{t\to 0}
\int_{B_v}
\widetilde{\eta_{d_1+1}(t)}
$$
is finite for $v\neq p$.
\item
We set
\begin{align*}
\overline{\eta_j}&=
\frac{x^{j-1}dx}{\overline{y_1}} \quad\text{ for }j=1,\dots, d_1, \\
\overline{\eta_{j+d_1}}&=
\frac{(x-\tilde\lambda)^{j-1}dx}
{\overline{y_2}} \quad\text{ for }j=1,\dots, d_2-1,
\end{align*}
where
\begin{align*}
& \overline{y_1}^3=(x-\tilde\lambda)^2
\prod_{k\neq i,i+1}^m(x-\lambda_k)^{a_k}, \\
& \overline{y_2}^3=
(x-\tilde\lambda)
\prod_{k\neq i,i+1}^m(x-\lambda_k)^{b_k}.
\end{align*}
Then we have
\begin{align*}
&\lim_{t\to 0}\int_{B_v}\eta_j(t)=
\int_{B_v}\overline{\eta_j} \quad 
\text{ for } j=1, \dots, d_1, \\
&\lim_{t\to 0}\int_{B_v}\widetilde{\eta_{j}(t)}=
\int_{B_v}\overline{\eta_{j-1}} 
 \quad 
\text{ for } j=d_1+2, \dots, d_1+d_2,
\end{align*}
for $v\neq p$.
\end{enumerate}
\end{proposition} 
\begin{proof}
We prove (1). We compute the integral
$$
\int_{B_p}\frac{dx}{y_2(t)}=
\omega^i(1-\omega)\int_{\lambda_i(t)}^{\lambda_{i+1}(t)}
\frac{dx}{(x-\lambda_i(t))^{2/3}(x-\lambda_{i+1}(t))^{2/3}
\prod_{k\neq i,i+1}(x-\lambda_k)^{b_k/3}}
$$
by the variable change $x=\tilde\lambda+t(\xi-\tilde\lambda)$. 
Since 
$x-\lambda_i(t)=t(\xi-\lambda_i)$,
$\xi$ varies $\lambda_i<\xi<\lambda_{i+1}$.
Using
$x-\lambda_k=\tilde\lambda-\lambda_k+t(\xi-\tilde\lambda)$,
we have
\begin{align*}
\lim_{t\to 0}
t^{1/3}\int_{B_p}\frac{dx}{y_2(t)}=&
\frac{\omega^i(1-\omega)}
{\prod_{k\neq i,i+1}(\tilde\lambda-\lambda_k)^{b_k/3}}
\int_{\lambda_i}^{\lambda_{i+1}}
\frac{d\xi}{(\xi-\lambda_i)^{2/3}(\xi-\lambda_{i+1})^{2/3}} \\
=&
\frac{-\omega^{i'}(1-\omega)}{(\lambda_i-\lambda_{i+1})^{1/3}
\prod_{k\neq i,i+1}(\tilde\lambda-\lambda_k)^{b_k/3}}
B(\frac{1}{3},\frac{1}{3}).
\end{align*}
We can similarly prove the rests.
\end{proof}

\subsection{Deformation of period matrices}
We define the period matrices $P_A=P_A(\Sigma,\Gamma, \C, M)$ and
$P_B=P_B(\Sigma,\Gamma, \C,M)$ associated to the
configuration of points $\Sigma=\{\lambda_1, \dots,\lambda_m\}$
and a marked binary tree $(\Gamma, \C,M)$ by
\begin{align*}
&P_A(\Sigma, \Gamma, \C, M)
=\left(\begin{matrix}
\int_{A_v}\eta_1& \dots &\int_{A_v}\eta_{d_1+d_2-2}
\end{matrix}\right)_{v\in V(\Gamma)}, \\
&P_B(\Sigma, \Gamma, \C, M)
=\left(\begin{matrix}
\int_{B_v}\eta_1& \dots &\int_{B_v}\eta_{d_1+d_2-2}
\end{matrix}\right)_{v\in V(\Gamma)}.
\end{align*}
We consider a family of configuration 
$\Sigma(t)$
defined in the last subsection.
Then we have a family of period matrices 
$P_B(\Sigma(t), \Gamma, \C,M)$.
We set 
$$
\Sigma'=\{\lambda_1, \dots, \lambda_{i-1},\tilde\lambda,
\lambda_{i+2},\cdots, \lambda_m\}.
$$
Then we have the following 
proposition.
\begin{proposition}
We have
\begin{align*}
&\lim_{t\to 0}t(\lambda_i-\lambda_{i+1})
\det(P_B(\Sigma(t),\Gamma,\C,M))^3 \\
=&
B^*(\frac{1}{3},\frac{1}{3})^3
\prod_{k\neq i,i+1}(\tilde\lambda-\lambda_k)^{-b_k}\cdot
\det(P_B(\Sigma',\Gamma', \C',M'))^3,
\end{align*}
where the marked binary tree $(\Gamma', \C',M')$
is obtained by the decomposition of the marked binary tree of
$(\Gamma, \C, M)$ illustrated in 
Figure \ref{simple degeneration}.
\end{proposition}
\subsection{Equi-distributed characteristics and theta constants}

Let $\lambda_1, \dots, \lambda_m$ be a configuration of points
in $\mathbb P^1$, $a_1/3, \dots, a_m/3$ be the branching index
of $\lambda_1, \dots, \lambda_m$ and
$(\Gamma,\C,M)$ be a marked binary tree as in the beginning of this section.
Let $\Sigma$ and $\overline{\Sigma}$
be sets given by
$$
\Sigma=\{\lambda_i\mid a_i=1\},\qquad
\overline{\Sigma}=\{\lambda_i\mid a_i=2\}.
$$
Let $\Lambda=\sum_i k_ie_i$
be an element in $\oplus_{i=1}^m\mathbb F_3e_i$. We set
$\Lambda_i,\overline{\Lambda_i}$ for $i=0,1,2$ by
$$
\Lambda_i=\{\lambda_j\in \Sigma \mid k_j=i\},\qquad
\overline{\Lambda_i}=\{\lambda_j\in \overline{\Sigma} \mid k_j=i\}.
$$
The element $\Lambda$
is in the kernel of
$\Pi:\oplus_{i=1}^m\mathbb F_3e_i\to \mathbb F_3$ defined in
(\ref{def of pi}) if and only if
\begin{equation}
\label{kernel pi}
\# \Lambda_1-\# \overline{\Lambda_1}+
2(\# \Lambda_2-\# \overline{\Lambda_2})\equiv 0 
\quad (\bmod 3).
\end{equation}
An element $\Lambda$ is said to be equi-distributed if and only if
$$
\# \Lambda_0-\# \overline{\Lambda_0}=
\# \Lambda_1-\# \overline{\Lambda_1}=
\# \Lambda_2-\# \overline{\Lambda_2}.
$$
By (\ref{kernel pi}), an equi-distributed element $\Lambda$ gives an element in 
$H_1(C,\mathbb Q/\mathbb Z)_{(1-\rho)}$.
If the colors of all terminals are white, an element $\Lambda\in \bold V$
is equi-distributed if and only if it satisfies the condition 
(\ref{equi-dist all white}).

We consider the decomposition of the marked binary tree 
as in Figure \ref{simple degeneration}.
For an element $\Lambda=\sum_j k_je_j\in H(\Gamma,\C)$, we define an element
$\Lambda'\in H(\Gamma',\C')$ by 
$\Lambda'=-(k_i+k_{i+1})e_{\tilde\lambda}+\sum_{j\neq i,i+1}k_je_j$.
The following lemma is easy to see.

\begin{lemma}
Let $\Lambda$ be an equi-distributed element in
$\oplus_i\mathbb F_3e_i$. Suppose that
$k_i\neq k_{i+1}$.
Then $\Lambda'$ is also equi-distributed. 
\end{lemma}

We consider the limit of the normalized period matrix.
The fiber of the stable curve at $\varepsilon=0$ becomes the union
of $C_1$ and $C_2$, where
\begin{align*}
&C_1:y^3=(x-\lambda_i)(x-\lambda_{i+1}), \\
&C_2:\eta^3=(\xi-\tilde\lambda)^2\prod_{j\neq i,i+1}(\xi-\lambda_j)^{a_i},
\end{align*}
 as in the last section.
Since the variation of Hodge structure is smooth, we can
compute the limit by considering logarithmic differentials.
We consider a continuous extension of the symplectic basis
$\{A_v,B_v\}_{v\in V(\Gamma)}$. 
Then $\{A_p,B_p\}$ is a symplectic basis of $C_1$ and
$\{A_v,B_v\}_{v\in V(\Gamma')}$ is that of $C_2$ in Figure \ref{simple degeneration}.
Let $\tau_i$ be the normalized period matrix of $C_i$ for $i=1,2$.
Then we have
$$
\lim_{\varepsilon\to 0}\tau=\left(\begin{matrix} 
\tau_1 & 0 \\ 0 & \tau_2 \end{matrix}\right).
$$
Under the above notations, we have the following proposition,
by the computation of Riemann constants in 
\S \ref{subsect:Riemann constants}.

\begin{proposition}
Let $\varrho$, $\varrho_1$ and $\varrho_2$ be the Riemann constants for
symplectic bases
$\{A_v,B_v\}_{v\in V(\Gamma)}$,
$\{A_p,B_p\}$ and
$\{A_v,B_v\}_{v\in V(\Gamma')}$. Then we have
$$
\lim_{\varepsilon\to 0}\vartheta(\tau(\varepsilon))
\left[\begin{matrix}\Lambda+\varrho
\end{matrix}
\right]^6=
\vartheta(\tau_1)\left[\begin{matrix}\frac{1-\varrho}{3}A_p+\varrho_1
\end{matrix}
\right]^6\cdot
\vartheta(\tau_2)\left[\begin{matrix}\Lambda'+\varrho_2
\end{matrix}
\right]^6.
$$
Here, we use the identification
(\ref{identification of homology}) and the remark just after it.
Moreover, we have 
\begin{align*}
\vartheta(\tau_1)
\left[\begin{matrix}\frac{1-\rho}{3}A_p+\varrho_1
\end{matrix}
\right]^6
&=\frac{3^{9/4}}{(2\pi)^6}\Gamma(1/3)^9\exp(\frac{-5\pi i}{12}).
\end{align*}
\end{proposition}
\begin{proof}
As for the last statement, we use the formula
(c.f. \cite{MTY}):
\begin{equation}
\label{variant of chowla-selberg}
\vartheta(\omega)\left[\begin{matrix}\frac{1}{6}, 
\frac{-1}{6}\end{matrix}\right]^6 
=\frac{3^{9/4}}{(2\pi)^6}\Gamma(1/3)^9\exp(\frac{-5\pi i}{12}).
\end{equation}
\end{proof}
\section{Thomae's formula  and Thomae's constant for arbitrary branching indices}
Now we are ready to give the statement of Thomae's theorem
for arbitrary indices. By the inductive structure, we also evaluate
the constants appeared in Thomae's theorem.

Let $C\to \mathbb P^1$ be the cyclic triple covering branching at
a set $\Sigma=\{\lambda_1, \dots, \lambda_m\}$ of $\mathbb P^1$ and
$a_1/3,\dots, a_m/3$ be branching indices at $\lambda_1,
\cdots,\lambda_m$, respectively. Let $(\Gamma, \C, M)$ be a marked
binary tree whose set of terminals is $\Sigma$, 
$\{A_v,B_v\}$ be the associated symplectic basis, and
$\Lambda$ be an equi-distributed element in $H(\Gamma,\C)$.
We set $\Lambda_i,\overline{\Lambda_i}$ as in the last section.
\begin{definition}
\begin{enumerate}
\item
Let $A_1,A_2$ be $\Lambda_i$ or $\overline{\Lambda_i}$ ($i=0,1,2$) and
$A_1\neq A_2$.
We define $(A_1,A_2)$ by
\begin{align*}
&(A_1,A_2)=\prod_{\lambda\in A_1, \mu\in A_2}
(\lambda-\mu),
\end{align*}
and $(A_1,A_1)$ by
$$
(A_1,A_1)=\prod_{i<j,\lambda_i,\lambda_j\in A_1}(\lambda_i-\lambda_j).
$$
\item
The difference product $\Delta(\Sigma,\Lambda)$ of an equi-distributed
characteristic $\Lambda$ is defined by
$$
\Delta(\Sigma,\Lambda)=\prod_{i=0}^2(\Lambda_i\Lambda_i)^3
(\overline{\Lambda_i}\overline{\Lambda_i})^3  
\cdot
\prod_{0\leq i<j \leq 2}(\Lambda_i\Lambda_j)
(\overline{\Lambda_i}\overline{\Lambda_j})\cdot
\prod_{0\leq i\neq j \leq 2}(\Lambda_i\overline{\Lambda_j})^2.
$$
\end{enumerate}
\end{definition}
\begin{theorem}
\label{main theorem}
\begin{enumerate}
\item
Under the above notations, we have
\begin{align*}
\vartheta(\tau)\left[\begin{matrix}\Lambda+\varrho\end{matrix}
\right]^6= \pm&\kappa_{\Lambda}\cdot
\Delta(\Sigma,\Lambda)
\cdot
\det(P_B(\Sigma,\Gamma, \C,M))^3,
\end{align*}
where $\kappa_{\Lambda}$ is an absolute constant depending only on
$\Lambda$. Moreover $\kappa_{\Lambda}^6$ does not depend on the
choice of $\Lambda$.
\item
The constant $\kappa_{\Lambda}^{6}$ is equal to
$\kappa^{6(v_1+v_2)}$, where 
\begin{equation}
\label{building block}
\kappa=((2\pi)^33^{3/4}\exp(\frac{11\pi i}{12}))^{-1}
\end{equation}
and $v_1$ and $v_2$ are the numbers of
white and black inner vertices, respectively.
\end{enumerate}
\end{theorem}

Let us prove the above theorem. 
We consider the degeneration studied in Section \ref{section degeneration}.
We use the notations $\Gamma',\C',\Sigma'$ and $\Lambda'$
as in the last section.
Let $\Lambda=\sum_i k_i e_i$ be an equi-distributed element
in $H(\Gamma, \C)$.
We prove the case:
$k_i=2,k_{i+1}=1$ and $a_i=a_{i+1}=1$.
The other cases can be similarly proved.
We prepare Lemma 
\ref{inductive difference product} and Proposition
\ref{recursive proposition}.
\begin{lemma}
\label{inductive difference product}
We have
$$
\frac{1}
{\prod_{i}(\tilde\lambda-\lambda_i)^{b_i}}
\cdot
\lim_{\varepsilon\to 0}
\Big[
\frac{\Delta(\Sigma(\epsilon),\Lambda)}{t(\lambda_i-\lambda_{i+1})}
\Big]
=\Delta(\Sigma',\Lambda').
$$
\end{lemma}
\begin{proof}
We consider a family of the difference product $(\Lambda_i(t),\Lambda_j(t))$
and so on. Then we have 
$$
\Lambda_2'=\Lambda_2-\{\lambda_i\},\quad
\Lambda_1'=\Lambda_1-\{\lambda_{i+1}\},\quad
\overline{\Lambda_0'}=\overline{\Lambda_0}\cup \{\tilde\lambda\}.
$$
Therefore, we have
\begin{align*}
&\lim_{t\to 0}(\Lambda_i(t)\Lambda_i(t))^3=
(\Lambda_i'\Lambda_i')^3(\Lambda_i',\tilde\lambda)^3 \quad i=1,2, \\
&\lim_{t\to 0}(\Lambda_0(t)\Lambda_i(t))=
(\Lambda_0\Lambda_i')(\Lambda_0,\tilde\lambda) \quad i=1,2, \\
&\lim_{t\to 0}\frac{1}{t}(\Lambda_1(t)\Lambda_2(t))=
(\Lambda_1'\Lambda_2')(\Lambda_1',\tilde\lambda)
(\Lambda_2',\tilde\lambda)(\lambda_{i}-\lambda_{i+1}),  \\
&\lim_{t\to 0}(\Lambda_i(t)\overline{\Lambda_0}(t))^2=
(\Lambda_i'\overline{\Lambda_0}')^2(\Lambda_i',\tilde\lambda)^{-2} 
(\tilde\lambda\overline{\Lambda_0})^2
\quad i=1,2, \\
&\lim_{t\to 0}(\Lambda_i(t)\overline{\Lambda_j}(t))^2=
(\Lambda_i'\overline{\Lambda_j})^2 
(\tilde\lambda\overline{\Lambda_j})^2
\quad (i,j)=(1,2) \text{ or } (2,1), \\
&\lim_{t\to 0}(\overline{\Lambda_0}(t)\overline{\Lambda_0}(t))^3=
(\overline{\Lambda_0}'\overline{\Lambda_0}')^3
(\tilde\lambda\overline{\Lambda_0})^{-3}, \\
&\lim_{t\to 0}(\overline{\Lambda_0}(t)\overline{\Lambda_i}(t))=
(\overline{\Lambda_0}'\overline{\Lambda_i})
(\tilde\lambda\overline{\Lambda_i})^{-1} \quad i=1,2.
\end{align*}
Since
$$
\prod_{k\neq i,i+1}(\tilde\lambda-\lambda_i)^{b_i}
=
(\Lambda_1',\tilde\lambda)^2(\Lambda_2',\tilde\lambda)^2
(\Lambda_0,\tilde\lambda)^2(\overline{\Lambda_0},\tilde\lambda)
(\overline{\Lambda_1},\tilde\lambda)(\overline{\Lambda_2},\tilde\lambda),
$$
we have the lemma.
\end{proof}
\begin{proposition}
\label{recursive proposition}
\begin{enumerate}
\item
The statement (1) of Theorem \ref{main theorem}
for $(\Gamma,\C, M)$ implies the same statement 
for $(\Gamma',\C',M')$.
\item
Suppose that the statement (1) of Theorem \ref{main theorem}
holds for $\Gamma$ and $\Gamma'$. Then we have the recursive relation
$$
\kappa_{\Lambda}=\pm \kappa_{\Lambda'}\kappa,
$$
where $\kappa$ is defined in (\ref{building block})
\end{enumerate}
\end{proposition}
\begin{proof}(1)
We assume that the statement holds for the set of
terminals $\Sigma(\epsilon)$ of $(\Gamma, \C,M)$
with $\epsilon \neq 0$.
We consider the following limit:
\begin{align}
\label{inductive and limit}
&\vartheta(\tau_1)[\Lambda_1+\varrho_1]^6
\vartheta(\tau')[\Lambda'+\varrho']^6  \\
\nonumber
= &
\lim_{\epsilon\to 0}\vartheta(\tau(\epsilon))[\Lambda+\varrho]^6 \\
\nonumber
 =&\pm \kappa_{\Lambda}\lim_{\varepsilon\to 0}
\Big[\Delta(\Sigma(\epsilon),\Lambda)
\cdot
\det (P_B(\Sigma(\epsilon),\Gamma,\C, M))^3 \Big]\\
\nonumber
 =&\pm \kappa_{\Lambda}
\lim_{\varepsilon\to 0}
\Big[
t(\lambda_i-\lambda_{i+1})\det (P_B(\Sigma(\epsilon),\Gamma,\C, M))^3
\Big]
\cdot
\lim_{\varepsilon\to 0}
\Big[
\frac{\Delta(\Sigma(\epsilon),\Lambda)}{t(\lambda_i-\lambda_{i+1})}
\Big]\\
\nonumber
=&\pm
\frac{B^*(\frac{1}{3},\frac{1}{3})^3\kappa_{\Lambda}}
{\prod_{i}(\tilde\lambda-\lambda_i)^{b_i}}
\cdot
\lim_{\varepsilon\to 0}
\Big[
\frac{\Delta(\Sigma(\epsilon),\Lambda)}{t(\lambda_i-\lambda_{i+1})}
\Big]
 \cdot\det (P_B(\Sigma',\Gamma',\C', M'))^3.
\end{align}
By Lemma \ref{inductive difference product}
and the equality (\ref{inductive and limit}), we have
$$
\vartheta(\tau')[\Lambda']^6
=\pm \kappa_{\Lambda'}
\Delta(\Sigma',\Lambda')
\cdot\det (P_B(\Sigma',\Gamma',\C', M'))^3,
$$
where
\begin{align*}
\kappa_{\Lambda'}&=
\pm \vartheta(\tau_1)
[\Lambda_1]^{-6}
B^*(\frac{1}{3},\frac{1}{3})^3\kappa_{\Lambda} \\
&=\pm\frac{(2\pi)^6}{3^{9/4}}\exp(\frac{5\pi i}{12})(\omega-1)^3
\frac{3^{3/2}}{(2\pi)^3}\kappa_{\Lambda} \\
&=\pm
(2\pi)^33^{3/4}\exp(\frac{11\pi i}{12})\kappa_{\Lambda}.
\end{align*}
\end{proof}
\begin{proof}[Proof of Theorem\ref{main theorem}]
We consider the case there exists a procedure of degenerations
\begin{equation}
\label{procedure}
\Lambda \to \Lambda_1+\Lambda^{(1)} \to
\Lambda_1+\Lambda_2+\Lambda^{(2)}\to \cdots \to
\Lambda_1+\cdots +\Lambda_g,
\end{equation}
with $\kappa_{\Lambda_i}\neq 0$.
In this case, by applying Proposition \ref{recursive proposition},
we have the theorem.

If all terminals are white, 
since we know that $\kappa_{\Lambda}^6$ is independent of the
choice of $\Lambda$, we have the theorem once we
know the existence of $\Lambda$ and the procedure
of degenerations (\ref{procedure}).
Actually we can choose such
$\Lambda$ and procedure (\ref{procedure}).

Now we prove the general case. For any marked binary graph and
any equi-distributed element $\Lambda$,
we can choose procedure
$$
\overline{\Lambda}\to \Lambda_1+\Lambda^{(1)}\to
\Lambda_1+\Lambda_2+\Lambda^{(2)}\to \cdots 
\Lambda_1+\cdots +\Lambda_k+\Lambda,
$$
where $\kappa_{\Lambda_i}\neq 0$
and $\overline{\Lambda}$ is a equi-distributed element 
for a marked binary graph
whose terminals are white. Since we know the theorem
on $\overline{\Lambda}$, we have the theorem for
$\Lambda$.
\end{proof}

\section{An example}
We consider the following marked binary tree in
Figure \ref{pict example} and 
an equi-distributed element
$-e_1+e_2+e_3-e_4 \in H(\Gamma,\C)$.
We assume that $\lambda_1<\lambda_2<\lambda_3<\lambda_4 \in \mathbb R$
to fix a branch of the third roots.

\begin{figure}[hbt]
\hskip 0.0in\includegraphics[scale=0.4]{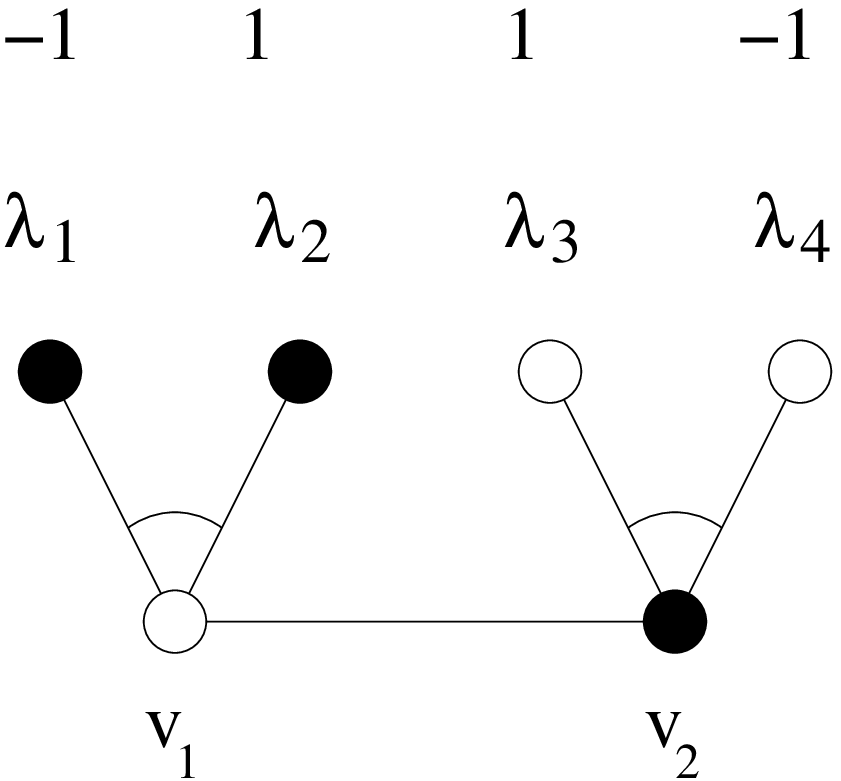}
\caption{Example}
\label{pict example}
\end{figure}

Then the branches of $y_1, y_2$ are given in Figure \ref{values on branches}.
\begin{figure}
\vskip 0.2in
\hskip 0in\includegraphics[scale=0.3]{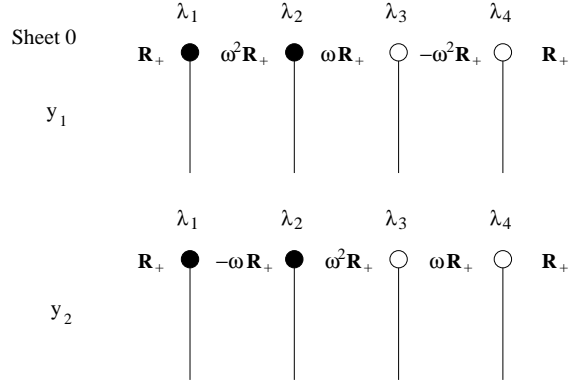}
\caption{Values on branches}
\label{values on branches}
\end{figure}
Thus we have $A=(a_{ij}), B=(b_{ij})$ with
\begin{align*}
&a_{11}=\int_{A_1}\frac{dx}{y_1}=(\omega^2-1)
\int_{\lambda_1}^{\lambda_2}\frac{dx}{\ ^3
\sqrt{(x-\lambda_1)^2(\lambda_2-x)^2(\lambda_3-x)(\lambda_4-x)}},
\\
&a_{21}=\int_{A_2}\frac{dx}{y_1}=(\omega^2-1)  
\int_{\lambda_3}^{\lambda_4}\frac{dx}{\ ^3
\sqrt{(x-\lambda_1)^2(x-\lambda_2)^2(x-\lambda_3)(\lambda_4-x)}},
\\
&a_{12}=\int_{A_2}\frac{dx}{y_2}=(1-\omega)
\int_{\lambda_3}^{\lambda_4}\frac{dx}{\ ^3
\sqrt{(x-\lambda_1)(x-\lambda_2)(x-\lambda_3)^2(\lambda_4-x)^2}},
\\
&a_{22}=\int_{A_1}\frac{dx}{y_2}= (1-\omega)
\int_{\lambda_1}^{\lambda_2}\frac{dx}{\ ^3
\sqrt{(x-\lambda_1)(\lambda_2-x)(\lambda_3-x)^2(\lambda_4-x)^2}},
\end{align*}

\begin{align*}
&b_{11}=\int_{B_1}\frac{dx}{y_1}=\omega^2\int_{A_1}
\frac{dx}{y_1},\quad
b_{21}=\int_{B_2}\frac{dx}{y_1}=\omega\int_{A_2}\frac{dx}{y_1},
\\
&b_{12}=\int_{B_1}\frac{dx}{y_2}=
\omega\int_{A_1}\frac{dx}{y_2},\quad
b_{22}=\int_{B_2}\frac{dx}{y_2}=\omega^2
\int_{A_2}\frac{dx}{y_2}.
\end{align*}
The above equi-distributed element is 
$\overline{A_1}-\overline{A_2}=\frac{1}{3}(A_1+A_2-B_1-B_2)$
and we have
\begin{align*}
\vartheta(\tau)[5/6,5/6,1/6,1/6]^6
=&\frac{1}{3\sqrt{3}(2\pi)^6}\det(B)^3 \exp(\frac{\pi i}{6})
\\
&(\lambda_2-\lambda_1)(\lambda_4-\lambda_3)
(\lambda_3-\lambda_1)^2(\lambda_4-\lambda_2)^2.
\end{align*}

\end{document}